\documentclass{amsart}
\usepackage{hyperref}
\hypersetup{bookmarksnumbered} 
\usepackage{amssymb,amsmath,amsthm}
\usepackage{graphicx}
\usepackage{ulem}
\usepackage{mathrsfs}
\usepackage[subrefformat=parens,labelformat=parens]{subfig}

\input xy
\xyoption{all}

\usepackage{verbatim}

\usepackage[dvipsnames]{xcolor}

\makeatletter

\let\c@table\c@figure
\makeatother

\newtheorem{theorem}{Theorem}[section]
\newtheorem*{BooneHigmanConjecture}{Boone--Higman conjecture}
\newtheorem*{HigmanEmbeddingTheorem}{Higman Embedding Theorem}
\newtheorem{proposition}[theorem]{Proposition}

\newtheorem{corollary}[theorem]{Corollary}

\theoremstyle{definition}

\newtheorem{question}[theorem]{Question}
\newtheorem{problem}[theorem]{Problem}

\theoremstyle{remark}
\newtheorem{remark}[theorem]{Remark}

\newcommand{\newword}[1]{\textbf{#1}}

\newcommand{\Z}{\mathbb{Z}}
\newcommand{\N}{\mathbb{N}}
\newcommand{\R}{\mathbb{R}}

\newcommand{\Q}{\mathbb{Q}}
\newcommand{\C}{\mathfrak{C}}

\newcommand{\Aut}{\mathrm{Aut}}
\newcommand{\Out}{\mathrm{Out}}

\newcommand{\VA}{V\!\hspace{0.1em}\mathcal{A}}
\newcommand{\A}{\mathcal{A}}

\newcommand{\Finfty}{\mathrm{F}_{\!\hspace{0.04em}\infty}}

\title{Progress around the Boone--Higman conjecture}

\author[J.~Belk]{James Belk}
\address{School of Mathematics and Statistics, University of Glasgow, University Place, Glasgow, G128QQ, Scotland.}
\email{\href{mailto:jim.belk@glasgow.ac.uk
}{jim.belk@glasgow.ac.uk
}}

\author[C.~Bleak]{Collin Bleak}
\address{School of Mathematics and Statistics, University of St Andrews, St Andrews, Scotland KY8 6ER.}
\email{\href{mailto:collin.bleak@st-andrews.ac.uk}{collin.bleak@st-andrews.ac.uk}}

\author[F.~Matucci]{Francesco Matucci}
\address{Dipartimento di Matematica e Applicazioni, Universit\`{a} degli Studi di Milano--Bicocca, Milan 20125, Italy.}
\email{\href{mailto:francesco.matucci@unimib.it}{francesco.matucci@unimib.it}}

\author[M.~C.~B.~Zaremsky]{Matthew C. B. Zaremsky}
\address{Department of Mathematics and Statistics, University at Albany (SUNY), Albany, NY 12222.} \email{\href{mailto:mzaremsky@albany.edu}{mzaremsky@albany.edu}}

\date{}  

\begin{document}

\begin{abstract}A conjecture of Boone and Higman from the 1970's asserts that a finitely generated group $G$ has solvable word problem if and only if $G$ can be embedded into a finitely presented simple group.  We comment on the history of this conjecture and survey recent results that establish the conjecture for many large classes of interesting groups.
\end{abstract}

\maketitle

\section{Introduction}

In 1961, Graham Higman proved that a finitely generated group has a computable presentation if and only if it embeds into a finitely presented group~\cite{Hig}, a celebrated result now known as the Higman embedding theorem.  In 1973, Higman and William Boone made the following conjecture~\cite{Boone,BoHi}, which can be viewed as a natural companion to Higman's theorem.

\begin{BooneHigmanConjecture}
A finitely generated group $G$ has solvable word problem if and only if it embeds into some finitely presented simple group.
\end{BooneHigmanConjecture}

This conjecture remains open fifty years later.  One direction is easy: as observed by Alexander Kuznetsov~\cite{Kuz}, every finitely generated subgroup of a finitely presented simple group must have solvable word problem (see Proposition~\ref{prop:ThompsonsStatement}). The converse, however, remains a mystery.

Note that the conjecture is open even for finitely presented groups.  Indeed, Christopher Clapham proved that every finitely generated group with solvable word problem embeds into a finitely presented group with solvable word problem~\cite{Clapham}, so it would suffice to prove the conjecture for finitely presented groups.

The Boone--Higman conjecture motivated much of the early work on finitely presented simple groups, including Higman's 1974 work on the Higman--Thompson groups~\cite{HigmanSimple}, Elizabeth Scott's 1984 construction of finitely presented simple groups that contain $\mathrm{GL}_n(\Z)$~\cite{Scott1,Scott2}, and Claas R\"over's construction in 1999 of a finitely presented simple group that contains Grigorchuk's group~\cite{Rov}.  Over the last two decades,  further work by many researchers including Matthew Brin \cite{BrinHigherDimensional,BrinPresentations,BrinBakers}, Volodymyr Nekrashevych~\cite{Nek,Nek2}, James Hyde~\cite{BHM,BHM2}, and the authors~\cite{BBMZ, BZ, ZaremskyTaste} has further expanded the class of known finitely presented simple groups, yielding significant progress towards the conjecture.

In this article, we trace through the history of the conjecture as well as more recent developments, and then summarize the state of the art on Boone--Higman embeddings for natural classes of finitely presented groups. A list of prominent groups for which the Boone--Higman conjecture is known to hold is given in Theorem~\ref{thrm:summary}, and a list of prominent groups for which the conjecture is open is given in Problem~\ref{prob:summary}.

\subsection*{Acknowledgments} We are grateful to Matt Brin, James Hyde, Carl-Fredrik Nyberg Brodda, Slobodan Tanushevski, and Xiaolei Wu for helpful conversations, and we would like to thank Martin Bridson and an anonymous referee for several helpful comments and suggestions. The third author is a member of the Gruppo Nazionale per le Strutture Algebriche, Geometriche e le loro Applicazioni (GNSAGA) of the Istituto Nazionale di Alta Matematica (INdAM) and gratefully acknowledges the support of the Funda\c{c}\~ao para a Ci\^encia e a Tecnologia  (CEMAT-Ci\^encias FCT projects UIDB/04621/2020 and UIDP/04621/2020) and of the Universit\`a degli Studi di Milano--Bicocca (FA project 2020-ATE-0006 ``Strutture Algebriche''). The fourth author is supported by grant \#635763 from the Simons Foundation.

\section{The Higman Embedding Theorem}

Much of the motivation for the Boone--Higman conjecture comes from the famous Higman embedding theorem.  Here we give a short exposition of this theorem and its consequences. 

Let $G$ be a finitely generated group.  We say that $G$ is \newword{computably\footnote{We will follow the common practice in logic and computer science of using the word ``computable'' in place of the more traditional ``recursive''.  See~\cite{Soare}.} presented} if there exists a presentation
\[
\langle s_1,\ldots,s_n \mid r_1,r_2,\ldots\rangle
\]
for $G$ whose relators\footnote{In what follows, if $G=\langle X\mid R\rangle$ is a group presentation, then $R\subseteq (X\cup X^{-1})^*$ and we call these words the \newword{relators} of the presentation, or less precisely, of $G$. If $w\in R$ is a relator, the relation $w= 1$ holds in $G$.} $r_1,r_2,\ldots$ are computably enumerable%
\footnote{One might expect this definition to require the set of relators $\{r_1,r_2,\ldots\}$ to be a computable set, instead of just computably enumerable.  However, if $G$ has a presentation $\langle s_1,\ldots,s_n \mid r_1,r_2,\ldots\rangle$ with a computably enumerable set of relators, then $\langle s_1,\ldots,s_n,t \mid t,tr_1,t^2r_2,t^3r_3,\ldots\rangle$ is a presentation for $G$ with a computable set of relators, so the two definitions are equivalent.}%
.  That is, there exists an algorithm that outputs the list of these relators one after another, continuing indefinitely if the list of relators is infinite.

Clearly any finitely presented group is computably presented.  More generally we have the following well-known proposition:

\begin{proposition}\label{prop:HigmanFirstHalf}
Any finitely generated subgroup of a finitely presented group is computably presented.
\end{proposition}
\begin{proof}
Let $G$ be a finitely presented group with finite generating set $S$, and let $H$ be a finitely generated subgroup of~$G$.  Fix words $w_1,\ldots,w_n$ over $S$ that represent the generators of~$H$ and their inverses.  Since $G$ is finitely presented, there exists an initial algorithm $\mathcal{A}$ that lists all possible consequences of the relators in~$G$, i.e.~it lists all words in $S$ that represent the identity in~$G$. 
Each time $\mathcal{A}$ outputs a word~$w$,  then a secondary algorithm can check whether $w$ decomposes exactly as a concatenation of words in $\{w_1,\ldots,w_n\}$, and if so, then this gives us a relator for $H$ (expressed in terms of the $w_i$) that can be output.  Note that any product in $H$ that is the identity in $H$ will have a corresponding concatenation of words from $\{w_1,\ldots,w_n\}$ that appears in the list of relators for $G$ (and so, it is an output of the first algorithm).  Thus, the secondary algorithm will output a relator for each product in $H$ that evaluates as the identity, resulting in a computable presentation for $H$.
\end{proof}

\begin{remark}
The proof of the above proposition works just as well if the larger group $G$ is computably presented.  Thus any finitely generated subgroup of a computably presented group is computably presented.
\end{remark}

\begin{remark}\label{rem:Listrelators}In the above proof, it might seem surprising that we were able to list all of the words for the identity in~$H$, but this is possible in any computably presented group, simply by listing all possible consequences of a computably enumerable set of relators.  Indeed, a group is computably presented if and only if the set of all words for the identity is computably enumerable.\end{remark}

In 1961, Higman proved the converse to Proposition~\ref{prop:HigmanFirstHalf}.

\begin{HigmanEmbeddingTheorem}[\cite{Hig}]Let $G$ be a finitely generated group.  Then $G$ is computably presented if and only if $G$ embeds into a finitely presented group.
\end{HigmanEmbeddingTheorem}

The proof of this theorem is a fairly complicated construction (see, e.g.~\cite[Chapter~12]{Rotman} for an in-depth development). 
Starting with a group $G$ and Turing machine $T$ that can write a set of relators for $G$, the proof uses the structure of~$T$ to construct a finitely presented group $\widetilde{G}$ that contains $G$ via a certain sequence of HNN extensions.

There are a few basic consequences of the Higman embedding theorem that will be important for us later.  First, we enlarge the class of groups that we call computably presented as follows: if $G$ is any countable group, we say that $G$ is \newword{computably presented} if $G$ has a presentation
\[
\langle s_1, s_2, s_3,\ldots \mid r_1,r_2,r_3,\ldots\rangle
\]
with an enumerated generating set $s_1,s_2,s_3,\ldots$ for which the relators $r_1,r_2,r_3,\ldots$ are computably enumerable. Note that this new definition is equivalent to the one given earlier in the case of a finitely generated group\footnote{If $G$ is finitely generated then some initial subset $s_1,\ldots,s_n$ of the given enumerated generating set generates.  Using the given presentation, we can compute words for each of the remaining generators in terms of these, and substituting these words into the relators leads to a computable presentation with generating set $s_1,\ldots,s_n$.}.

\begin{corollary}\label{cor:InfinityGeneratedHigman}
Every countable, computably presented group $G$ embeds into a finitely presented group.
\end{corollary}
\begin{proof}
It is proven in \cite{HNN} that every countable group $G$ can be embedded into a group $G^*$ that is generated by two elements. Furthermore, this construction has the property that if $G$ is computably presented in the more general sense, then $G^*$ is in the original sense.  The corollary now follows from the Higman embedding theorem. \end{proof}

\begin{remark}The converse of Corollary~\ref{cor:InfinityGeneratedHigman} need not hold.  That is, an arbitrary countable subgroup $H$ of a finitely presented group $G$ need not be computably presented.  The reason the proof of Proposition~\ref{prop:HigmanFirstHalf} does not go through is that $H$ might not be a computably generated subgroup, i.e.~there might not exist an algorithm that lists a generating set for $H$ as words in the generators for~$G$.  An  algorithmic characterization of arbitrary subgroups of finitely presented groups has been given by Sacerdote~\cite{Sac1}.
\end{remark}

\begin{corollary}\label{cor:HigmanAbelian}
Every countable abelian group embeds into a finitely presented group.
\end{corollary}
\begin{proof}
Consider the group $G=\bigl(\bigoplus_\omega \Q\bigr) \oplus \bigl(\bigoplus_\omega \Q/\Z\bigr)$.  It is not hard to give an explicit computable presentation of this group, so by Corollary~\ref{cor:InfinityGeneratedHigman} it embeds into a finitely presented group.  Thus, it suffices to prove that every countable abelian group embeds into~$G$.

This requires some theory of countable abelian groups.  An abelian group $A$ is called \newword{divisible} if for every $n\in\N$ and $a\in A$ there exists an element $a'\in A$ so that $na'=a$.  It is easy to prove~\cite[Theorem~4.1.4]{Fuchs} that every countable abelian group embeds into a countable divisible abelian group. Furthermore, every divisible abelian group is a direct sum of copies of $\Q$ and copies of quasicyclic groups, that is, groups of the form $\Z(p^\infty) = \bigcup_{k=1}^\infty \Z/p^k\Z$ for different primes~$p$~\cite[Theorem~4.3.1]{Fuchs}. 
 Since each $\Z(p^\infty)$ embeds into $\Q/\Z$, we conclude that every countable divisible abelian group embeds into~$G$, and the result follows.
\end{proof}

In particular let us emphasize that the additive group $\Q$ of the rational numbers embeds into a finitely presented group.  Higman's construction can be carried out explicitly (see~\cite{Mik}), but the resulting presentation has a very large number of generators and relators.  Higman was for many years interested in finding a more explicit and natural example of a finitely presented group that contains~$\Q$ \cite{Johnson}, and in 1999 Martin Bridson and Pierre de la Harpe submitted this question to the Kourovka notebook of unsolved problems in group theory~\cite[Problem~14.10]{Kourovka}.  A solution to this problem has recently been discovered by the first author, James Hyde, and the third author (see \cite{BHM}), and indeed, an explicit example of a finitely presented group that contains every countable abelian group is now known (see Section~\ref{sec:VA}).

Bridson and de la Harpe also asked about the computably presented groups $\mathrm{GL}_n(\Q)$, and this part of their question remains open.

\begin{problem}\label{prob:GLQ}
For $n\geq 2$, find an explicit and natural example of a finitely presented group that contains~$\mathrm{GL}_n(\Q)$.
\end{problem}

Another surprising corollary of the Higman embedding theorem is the following.

\begin{corollary}
There exists a finitely presented group $G$ into which every finitely presented group embeds.
\end{corollary}
\begin{proof}
It is easy to make an algorithm that outputs a list of all possible finite presentations of groups.  Let $G_1,G_2,\ldots$ be the resulting groups. 
 Then the direct sum $\bigoplus_{n\in\N} G_n$ is computably presented, so by Corollary~\ref{cor:InfinityGeneratedHigman} it embeds into some finitely presented group~$G$.
\end{proof}

No explicit example of such a group $G$ is known, though it is known that there exists an example with 21 relations~\cite{Val}.  Note that such a group $G$ must have unsolvable word problem, among other pathological properties.

\section{The Boone--Higman conjecture}\label{sec:BH-Background}

In this section we describe the motivation for the Boone--Higman conjecture, some of the consequences if the conjecture is true, and some partial general results.

\subsection{Motivation}
A natural question arising from the Higman embedding theorem is whether there might be other theorems of this type.  That is, given an algorithmic property $P$ of groups, can we find natural class $K$ of groups such that
\[
G\text{ has property }P\quad\Leftrightarrow\quad G\text{ embeds into a group from class }K.
\]
One obvious property $P$ to consider is having solvable word problem. Indeed, given a group $G$ with finite generating set $S$, let $W$ be the set of words over $S$ that represent the identity in $G$.  Then:
\begin{enumerate}
    \item $G$ is computably presented if and only if $W$ is computably enumerable.\smallskip
    \item $G$ has solvable word problem if and only if $W$ is computable.
\end{enumerate}
See Remark~\ref{rem:Listrelators} for an explanation of statement~(1), and see \cite{MillerSurvey} for a wealth of background on solvability of the word problem and other decision problems for groups. From this point of view, it seems very reasonable to hope for a version of the Higman embedding theorem for groups with solvable word problem.

But what class of groups should replace finitely presented groups in the statement of the theorem? Boone and Higman found the following proposition illuminating.

\begin{proposition}\label{prop:ThompsonsStatement}
Every finitely presented simple group has solvable word problem.
\end{proposition}
\begin{proof}
Let $G=\langle s_1,\ldots,s_m\mid r_1,\ldots,r_n\rangle$ be a finitely presented simple group.  Given a word $w$ over $s_1,\ldots,s_m$, we will run two algorithms $A_1$ and $A_2$ in parallel, the first searching for a proof that $w=1$, and the second searching for a proof that~$w\ne 1$, terminating as soon as one of the algorithms is successful.

The first algorithm $A_1$ exists for any finitely presented group (or any computably presented group).  All it does is derive all possible consequences of the relators $r_1,\ldots,r_n$, and compares these words one-at-a-time against $w$.  If $w=1$ in the group, then it will eventually appear as a consequence of the relators and the algorithm will terminate. If $w\ne 1$, then $A_1$ never terminates.

The second algorithm $A_2$ only works for simple groups.  The idea is that if $w\ne 1$, then adding the relator $w=1$ to the presentation for $G$ must yield a trivial quotient group.  So all we do is start with the relators $r_1,\ldots,r_n$ together with the relator $w=1$ and try to prove that $s_i=1$ for all~$i$.  Again, if $w\ne 1$, then this algorithm $A_2$ eventually finds such a proof, but if $w=1$ then $A_2$ never terminates.  Together, $A_1$ and $A_2$ constitute a solution to the word problem in~$G$.
\end{proof}

Boone and Higman first became aware of Proposition~\ref{prop:ThompsonsStatement} when it was presented by Richard~J.\ Thompson during a 1969 conference\footnote{The ``Decision Problems in Group Theory'' CODEP conference at the University of Irvine, California in September 1969. See~\cite{BCL} for the proceedings.} talk on his work with Ralph~McKenzie~\cite{McTh}.  It was later discovered that the same observation had been made earlier by Kuznetsov in a short 1958 article published in Russia~\cite{Kuz}.  It seems likely that Boone and Higman formulated their conjecture during that conference, however the first appearance of the conjecture in print appears to be in the proceedings of a 1973 conference ~\cite{Boone}, and then shortly after in a 1974 paper~\cite{BoHi}.

Boone and Higman recognized Thompson's argument as a special case of a principle in logic that a complete deductive system with finitely many axioms is decidable.  That is, there is an algorithm to determine whether a given statement in such a system follows from the axioms.  The algorithm works the same way as above: on the one hand, we search for a proof of the given statement from the axioms, while on the other hand we assume the given statement and search for a contradiction.  The idea here is that a group presentation is analogous to an axiomatic system, with the relators being axioms.  From this point of view:
\begin{itemize}
    \item A presentation of the trivial group is analogous to an inconsistent system, since every relation holds in the trivial group, and every statement is provable in an inconsistent system.\smallskip
    \item A simple group is analogous to a complete system, since adding new relations to a simple group makes it trivial, while adding a new axiom to a complete system makes it inconsistent.\smallskip
    \item A group with solvable word problem is analogous to a decidable system, since in both cases there is an algorithm to determine whether a given statement holds.
\end{itemize} 
Based on these considerations, the following conjecture of Boone and Higman seems very natural.

\begin{BooneHigmanConjecture}
A finitely generated group $G$ has solvable word problem if and only if it embeds into some finitely presented simple group.
\end{BooneHigmanConjecture}

Note that if $G$ has solvable word problem and $H$ is a finitely generated subgroup of~$G$, then $H$ must have solvable word problem.  Thus, it follows from Proposition~\ref{prop:ThompsonsStatement} that every finitely generated group that embeds into a finitely presented simple group has solvable word problem.

\subsection{Boone--Higman for countable groups}

As with the Higman embedding theorem, the Boone--Higman conjecture would imply a more general embedding statement about countable groups.  In particular, we say that a countable group $G$ has \newword{solvable word problem} with respect to an infinite, enumerated generating set $s_1,s_2,\ldots$ if there exists an algorithm to determine whether a given word in these generators represents the identity in~$G$.  Unlike with finitely generated groups, this property may depend on the generating set.

We begin by stating a generalization of Proposition~\ref{prop:ThompsonsStatement} to this setting.

\begin{proposition}\label{prop:ModifyThompsonProof}
If $G$ is a countable simple group with a computable presentation, then $G$ has solvable word problem with respect to the corresponding enumerated generating set. 
\end{proposition}
\begin{proof}Since $G$ might have infinitely many generators $s_1,s_2,\ldots$, the proof of Proposition~\ref{prop:ThompsonsStatement} does not go through as stated, since we cannot check algorithmically that $w=1$ implies that $s_i=1$ for all~$i$.  However, it suffices to fix any one word $v$ that represents a nontrivial element of~$G$, and then replace $A_2$ with an algorithm that checks whether adding the relation $w=1$ implies that $v=1$.  Note that we do not need any algorithm or procedure to somehow find such a word~$v$ from the given computable presentation---it suffices to know that there exists such a~$v$, and hence the desired algorithm~$A_2$ exists as well.
\end{proof}

For the following proposition, we say that a countable group has \newword{a solvable word problem} if it has solvable word problem with respect to some finite or infinite generating set.  Note that such a group is necessarily computably presented.

\begin{proposition}
Every countable group with a solvable word problem embeds into a finitely presented group with solvable word problem.  Furthermore, the following groups can be embedded into finitely presented groups with solvable word problem:
\begin{enumerate}
    \item Countable abelian groups.\smallskip
    \item The group $\mathrm{GL}_n(\Q)$ for all $n\geq 2$.
\end{enumerate}
\end{proposition}
\begin{proof}
As in the proof of Corollary~\ref{cor:InfinityGeneratedHigman}, the Higman--Neumann--Neuman construction in \cite{HNN} embeds every countable group with a solvable word problem into a two-generated group with solvable word problem.  Clapham \cite{Clapham} proved that every finitely generated group with solvable word problem embeds into a finitely presented group with solvable word problem, and the main statement follows. Since $\bigl(\bigoplus_\omega \Q\bigr)\oplus \bigl(\bigoplus_\omega \Q/\Z\bigr)$ has solvable word problem, statement (1) follows (as in the proof of Corollary~\ref{cor:HigmanAbelian}). For statement (2), it it easy to enumerate a generating set for $\mathrm{GL}_n(\Q)$, and the word problem is solvable by direct matrix computations.
\end{proof}

\begin{remark}Not every countable abelian group has a solvable word problem.  For example, if $S$ is a set of prime numbers that is not computably enumerable, then $G = \bigoplus_{s\in S}\Z/s\Z$ has unsolvable word problem with respect to any enumerated generating set, since otherwise we could enumerate $S$ by listing all primes that divide the orders of the elements of~$G$.
\end{remark}

\begin{corollary}
If the Boone--Higman conjecture holds, then, every countable group with a solvable word problem embeds into a finitely presented simple group.  Furthermore, every countable abelian group embeds into a finitely presented simple group, as does\/ $\mathrm{GL}_n(\Q)$ for all $n\geq 2$.\hfill\qedsymbol
\end{corollary}

As mentioned above, it is now known that every countable abelian group embeds into a finitely presented simple group (see Section~\ref{sec:VA}), but it remains an open question whether $\mathrm{GL}_n(\Q)$ embeds into a finitely presented simple group.

\subsection{Partial results}
Though Boone and Higman were not able to prove their conjecture, they did prove a partial result.

\begin{theorem}[Boone--Higman 1974 \cite{BoHi}]\label{thm:BooneHigman}A finitely generated group has solvable word problem if and only if embeds into a computably presented simple group.
\end{theorem}
\begin{proof}[Sketch of Proof]
By Proposition~\ref{prop:ModifyThompsonProof}, any computably presented simple group has solvable word problem, so any finitely generated subgroup of such a group has solvable word problem as well.

For the converse, observe that a group is simple if and only if the normal closure of each nontrivial element is the whole group. The proof involves a trick for inductively building groups that satisfy this condition. 

Suppose $G$ is a group and $v$ and $w$ are words for nontrivial elements of~$G$, and consider the group
\[
G' = \bigl\langle G,x,t \;\bigr|\; (vv^x)^t = v^xw \bigr\rangle
\]
where $a^b$ denotes $b^{-1}ab$.  Note that $vv^x$ and $v^xw$ are infinite-order elements of the free product $G*\langle x\rangle$, so $G'$ is an HNN extension of $G*\langle x\rangle$, where $t$ conjugates the infinite cyclic group $\langle vv^x\rangle$ to the infinite cyclic group $\langle v^xw\rangle$.  It follows that the natural homomorphism $G\to G'$ is an embedding.  However, since $w=(v^x)^{-1}(vv^x)^t$, the group $G'$ has the property that $w$ lies in the normal closure of~$v$.

Now, if $G$ is any countable group with a solvable word problem (with respect to some generating set),  then there exists an algorithm than enumerates all pairs $(v_i,w_i)$ of words for nontrivial elements of~$G$. Let
\[
G^* = \bigl\langle G,x,t_1,t_2,\ldots \;\bigr|\; (v_iv_i^x)^{t_i}=v_i^xw_i\text{ for all }i \bigr\rangle.
\]
As in the previous paragraph, $G$ embeds into $G^*$, and by construction the normal closure in $G^*$ of any nontrivial element of $G$ contains all of~$G$. Furthermore, it is not hard to show that $G^*$ again has solvable word problem.  Then the union of the sequence
\[
G \leq G^* \leq G^{**} \leq G^{***}\leq \cdots
\]
is a computably presented simple group that contains~$G$.
\end{proof}

Of course, the simple group constructed in the above proof is very far from being finitely presented. However, applying the Higman embedding theorem gives us the following corollary.

\begin{corollary}\label{cor:SimpleSubgroup}
A finitely generated group has solvable word problem if and only if it embeds into a simple subgroup of a finitely presented group.
\end{corollary}
\begin{proof}
If a finitely generated group has solvable word problem, then it embeds into a computably presented simple group by Theorem~\ref{thm:BooneHigman}, and by Corollary~\ref{cor:InfinityGeneratedHigman} this embeds into a finitely presented group.

The converse involves a modification of the proof of Proposition~\ref{prop:ModifyThompsonProof}.  Suppose that $G\leq S\leq H$, where $G$ is finitely generated, $S$ is simple, and $H$ is finitely presented, and without loss of generality suppose that each generator for $G$ is also a generator for~$H$.  By Proposition~\ref{prop:HigmanFirstHalf}, the group $G$ is computably presented, so all we need is an algorithm $A_2$ that recognizes when a given word $w$ in $G$ represents a non-identity element.  Note that if $w$ is such a word, then adding the relation $w=1$ to $H$ gives a proper quotient of $H$ in which $S$ collapses to the trivial group.  Thus, it suffices to fix any word $v$ in $H$ that represents a nontrivial element of~$S$, and have $A_2$ be the algorithm that searches for a proof that $v=1$ using the relations in~$H$ together with the relation $w=1$.
 \end{proof}

Thompson was later able to improve Boone and Higman's Theorem \ref{thm:BooneHigman}. 
\begin{theorem}[Thompson 1980 \cite{Tho}]\label{thm:Tho}A finitely generated group has solvable word problem if and only if it embeds into a finitely generated, computably presented simple group.\hfill\qedsymbol
\end{theorem}

As in Corollary~\ref{cor:SimpleSubgroup}, it follows that a finitely generated group has solvable word problem if and only if it embeds into a finitely generated simple subgroup of a finitely presented group. See Corollary~\ref{cor:BZCor} below for an alternative, and more specific, proof of Thompson's result, as well as~\cite[Appendix~A]{DarSte}.

Finally, we should mention that various other decision problems for groups are now known to have associated embedding theorems, including the conjugacy problem, the order problem, the power problem, the subgroup membership problem, and the isomorphism problem \cite{Belegradek,BoHi2,Sac2}.  In addition, analogs of the Boone--Higman theorem have been proven for lattice-ordered groups~\cite{Glass} and left-orderable groups \cite{BludovGlass,DarSte}, as well as certain other algebraic structures such as rings of characteristic~$p$, magmas, loops, and lattices~\cite{EMN}, but the corresponding statement for semigroups appears to be much more difficult~\cite{Maltcev}.

\section{The hunt for ``good'' containers}
A strategy towards proving the Boone--Higman conjecture is to find a nice family of finitely presented simple groups to serve as targets for embeddings of groups with solvable word problem.  In this section we discuss how this has proceeded for various groups of interest. We should mention that there cannot exist a finitely presented simple group that simultaneously contains copies of all finitely presented groups with solvable word problem\footnote{This follows from the Boone--Rogers theorem that there is no single algorithm that takes as input a finite presentation of a group $G$ with solvable word problem and a word $w$ and determines whether $w$ is the identity in $G$~\cite{BoRo}.  If we had a single finitely presented simple group $S$ that contained every finitely presented group with solvable word problem, then we would get such an algorithm by searching for a homomorphism from $G$ to $S$ that maps $w$ to a non-identity element of~$S$ (and simultaneously searching for a proof that $w$ is the identity using the relations of~$G$).}. Thus, generally speaking, the hunt is for families of Boone--Higman containers, rather than a single universal Boone--Higman container.

\subsection{The Thompson groups $\boldsymbol{F}$, $\boldsymbol{T}$, and $\boldsymbol{V}$.}
At the time Boone and Higman made their conjecture, the only known infinite, 
finitely presented simple groups were Thompson's groups $T$ and $V$, and the 
Higman--Thompson groups $V_{d,r}$, which are generalizations of Thompson's group 
$V$ created by Higman \cite{HigmanSimple}.
However, all of the groups $V_{d,r}$ embed into~$V$.

In the 1960's, Richard J.\ Thompson introduced three finitely presented groups $F\leq T \leq V$, now known collectively as \newword{Thompson's groups}.  Each group has a natural action on a compact space, with $F$ acting on a closed interval, $T$~acting on the circle, and $V$ acting on the Cantor set.  At the time of their introduction, $T$ and $V$ were the only known examples of infinite, finitely presented simple groups. Even today, most known infinite, finitely presented simple groups are variations on Thompson's groups, with arguably the only exceptions being the finitely presented simple groups introduced by Marc Burger and Shahar Mozes~\cite{BuMo1,BuMo2} and the Kac-Moody groups used by Pierre-Emmanuel Caprace and Bertrand R\'emy \cite{CR1,CR2}.

Though it will not be relevant to this discussion, Thompson's groups also have many other remarkable algebraic and geometric properties (see \cite{CFP} for a general introduction to these groups). Since $F$ and $T$ both embed into $V$, only the latter group will be interesting for our discussion of the Boone--Higman conjecture.

\newcommand{\cone}[1]{#1\C}

Thompson's group $V$ can be defined as follows.  Let $\C=\{0,1\}^\omega$ be the Cantor set of all infinite binary sequences.  Given a finite binary sequence $\alpha$, the corresponding \newword{cone} $\cone{\alpha}$ is the subset of $\C$ consisting of infinite sequences that start with~$\alpha$. Given two partitions $\cone{\alpha_1},\ldots,\cone{\alpha_n}$ and $\cone{\beta_1},\ldots,\cone{\beta_n}$ of $\C$ into cones, we can define a homeomorphism $f\colon \C\to\C$ that maps each $\cone{\alpha_i}$ to $\cone{\beta_i}$ in the canonical fashion, i.e.~$f(\alpha_i\psi)=\beta_i\psi$ for all infinite binary sequences~$\psi$.  The collection of all homeomorphisms of $\C$ defined in this way is Thompson's group~$V$. 

\begin{theorem}Thompson's group $V$ is simple and finitely presented.
\end{theorem}
\begin{proof}[Sketch of Proof]
As observed by the second author and Martyn Quick in~\cite{BleakQuick}, Thompson's group $V$ is in some sense a natural generalization of the finite symmetric groups, though there is not a distinction in $V$ between even and odd permutations. In particular, if $\cone{\alpha}$ and $\cone{\beta}$ are disjoint cones in~$\C$, define the corresponding \newword{transposition} to be the element $(\alpha\;\beta)$ of $V$ that swaps $\cone{\alpha}$ with $\cone{\beta}$ via canonical homeomorphisms and is the identity elsewhere. It is not difficult to show that all of the transpositions other than $(0\;1)$ are conjugate in $V$, and that this conjugacy class generates~$V$.

To prove that $V$ is simple, let $N$ be a nontrivial normal subgroup of~$V$, and let $n$ be a nontrivial element of~$N$.  Then there exist disjoint cones $\cone{\alpha}$ and $\cone{\beta}$ with $(\alpha,\beta)\ne (0,1)$ so that $n$ maps $\cone{\alpha}$ to $\cone{\beta}$ in the canonical fashion.  Now the commutator \[
(\alpha0\;\,\alpha1)\,n\,(\alpha0\;\,\alpha1)\,n^{-1} = (\alpha 0\;\alpha 1)\,(\beta 0\;\beta 1)
\]
lies in $N$, and conjugating by $(\alpha 1\;\,\beta 0)$ gives us that $(\alpha\;\beta)\in N$.  Conjugating by elements of $V$ catches all transpositions other than $(0\;1)$. Since these generate~$V$, it follows that $N=V$, which proves that $V$ is simple.

Known arguments for the finite presentability of $V$ involve either constructing a complex on which $V$ acts (e.g.~\cite{BrownFiniteness} or \cite{FarleyTV}) or going through nontrivial computations involving relators.  For example, the second author and Quick \cite{BleakQuick} prove that $V$ has an infinite presentation with the transpositions as generators, and relations of the form:
\begin{enumerate}
    \item $(\alpha\;\beta)^2 = 1$.\smallskip
    \item $(\alpha\;\beta) = (\alpha 0\;\,\beta 0)(\alpha 1\;\,\beta 1)$.\smallskip
    \item $(\alpha\;\beta)^{(\gamma\;\delta)} = (\epsilon\;\zeta)$ for certain triples $(\alpha\;\beta)$, $(\gamma\;\delta)$, and $(\epsilon\;\zeta)$. 
\end{enumerate}
They then use explicit computations to reduce this to a finite presentation with two generators and seven relators.
\end{proof}

\begin{remark}
Inspired by Thompson's work, Higman in~\cite{HigmanSimple} considered the groups~$V_d$, which are defined similarly to $V$, but act on the Cantor space $\{0,1,\ldots,d-1\}^\omega$ for some $d\geq 2$. More generally, he also defined the groups $V_{d,r}$ that act on the Cantor space $\{0,1,\ldots,r-1\}\times\{0,1,\ldots,d-1\}^\omega$ for $r\geq 1$ and $d\geq 2$ (he uses the notation $G_{n,r}$ for these).  Higman proved that all of these groups are finitely presented.  Surprisingly, the groups $V_d$ and $V_{d,r}$ are only simple when $d$ is even, though when $d$ is odd they have simple subgroups of index two.  As mentioned above, Higman also proved that all of these groups embed into Thompson's group~$V$, making them largely irrelevant to the Boone--Higman conjecture.
\end{remark}

We should mention that $F$, $T$, $V$, and all the $V_{d,r}$ are not only finitely presented but even have type $\mathrm{F}\!_\infty$ \cite{BrGe,BrownFiniteness}. Recall that a group has \newword{type $\mathrm{F}\!_n$} if it has a classifying space with finite $n$-skeleton, and has \newword{type $\mathrm{F}\!_\infty$} if it has type $\mathrm{F}\!_n$ for all $n$. In particular type $\mathrm{F}\!_1$ is equivalent to finite generation and $\mathrm{F}\!_2$ is equivalent to finite presentability. See, e.g.~\cite[Chapter~7]{Geog} for more background on these and other finiteness properties.

The following theorem lists some other known subgroups of~$V$.

\begin{theorem}\label{thm:GroupsVContains}
The following countable groups embed into Thompson's group~$V$:
\begin{enumerate}
    \item Any finite group,  free abelian group, or nonabelian free group;\smallskip
    \item $\bigoplus_\omega V$, and hence any countable direct sum of subgroups of~$V$;\smallskip
    \item {\normalfont(Higman 1974~\cite[Theorem~6.6]{HigmanSimple})} Any countable group (such as\/ $\Q/\Z$, and the finitely supported infinite symmetric group) that is an ascending union of finite groups;\smallskip
    \item {\normalfont(R\"over 1999~\cite{RoverPhD})} Any group that has a subgroup of finite index that embeds into~$V$; the Houghton groups~$H_n$; any free product of finitely many finite groups;\smallskip
    \item {\normalfont(Bleak--Salazar-D{\'\i}az 2013~\cite{BleakSalazar})} Any restricted wreath product\/ $V\wr (G_1\times H_1)$ or free product\/ $(G_1\times H_1)*(G_2\times H_2)$, where each $G_i$ is finite and each $H_i$ is either a free group or\/~$\Q/\Z$.\hfill\qedsymbol
\end{enumerate}
\end{theorem}

Groups of type (3) above are known as ``countable locally finite groups'', and have attracted much attention in the literature~\cite{KeWe}.

In addition, there are now several known obstructions to embedding into~$V$.

\begin{theorem}[Higman 1974~\protect{\cite{HigmanSimple}}]\label{thm:NoSL3ZInV}
The group\/ $\mathrm{SL}_3(\Z)$ does not embed into~$V$.  Furthermore, every torsion-free nilpotent group that embeds into $V$ is free abelian.  In particular,\/ $\Q$ does not embed into $V$, and neither does any non-cyclic subgroup of~$\Q$ such as\/ $\Z\bigl[\frac12\bigr]$.
\end{theorem}
\begin{proof}[Sketch of proof]
Higman proves that if $H\leq V$ is a torsion-free abelian group of finite rank (i.e.\ $H\otimes\Q$ is finite-dimensional), then:
\begin{enumerate}
    \item $H$ is free abelian,\smallskip
    \item The centralizer $C_V(H)$ has finite index in the normalizer $N_V(H)$, and\smallskip
    \item There exists a direct factor $B$ of $N_V(H)$ and a free abelian subgroup $C\leq B$ such that $C$ has finite index in~$HB$.
\end{enumerate}
Combining this with known properties of $\mathrm{SL}_3(\Z)$ and torsion-free nilpotent groups yields the desired results.

Higman also proves that if $f\in V$ has infinite order, then there exist only finitely many $k$ such that $f$ has a $k$th root.  This gives an alternative proof that non-cyclic subgroups of $\Q$ do not embed into~$V$.
\end{proof}

For the following theorem, recall that a group $G$ has \newword{Burnside type} if it is infinite, finitely generated, and every element has finite order.  Burnside asked in 1902 whether such groups existed~\cite{Burnside}, and examples were not found until 1964 by Golod and Shafarevich~\cite{Golod}.

\begin{theorem}[R\"over 1999 \cite{RoverPhD,Rov}]
Thompson's group $V$ has no subgroups of Burnside type.\hfill\qedsymbol
\end{theorem}

Note that there are many groups of Burnside type with solvable word problem.  The most prominent examples are self-similar groups such as the Grigorchuk group and the Gupta--Sidki groups, which we will discuss in the next subsection. However, it is also known that for any $m\geq 2$ and any $n\geq 2^{48}$ that is either odd or divisible by $2^9$, the free Burnside group
\[
B(m,n) = \langle s_1,\ldots,s_m \mid w^n=1\text{ for every word }w\text{ over }s_1,\ldots,s_m\rangle
\]
is infinite and has solvable word problem~\cite{Ivanov}.  No embeddings of these groups $B(m,n)$ into finitely presented simple groups are currently known, though explicit embeddings of $B(m,n)$ into finitely presented groups have been given by Ivanov~\cite{Ivanov2}.

Some recent results give much stronger obstructions to embedding into~$V$.

\begin{theorem}The following groups do not embed into Thompson's group $V$:
\begin{enumerate}
    \item {\normalfont(Bleak--Salazar-D\'\i az 2013 \cite{BleakSalazar})} The free product\/ $\Z^2*\Z$;\smallskip
    \item {\normalfont(Corwin 2013 \cite{Corwin})} The restricted wreath product\/ $\Z\wr\Z^2$;\smallskip
    \item {\normalfont(\cite{BleakMatucciREU}, \cite{BMN},  and \cite{BCR})} Any group with cyclic subgroups that are distorted, such as finitely generated nilpotent groups that are not virtually abelian, or Baumslag--Solitar groups $\mathrm{BS}(m,n)$ for\/ $|m|\ne |n|$.
\end{enumerate}
\end{theorem}

See~\cite{BCR} for a similar list of such results.

Statement (1) above is particularly strong, since many groups of interest are known to contain $\Z^2*\Z$.  For example, it follows from (1) that the only right-angled Artin groups that embed into $V$ are direct products of free groups~\cite{CoHa}.  Also, if $B_n$ is the braid group on $n\geq 4$ strands with standard generating set $\sigma_1,\ldots,\sigma_{n-1}$, then the subgroup $\bigl\langle\sigma_1^2,\sigma_2^2,\sigma_3^2\bigr\rangle$ is isomorphic to $\Z^2*\Z$~\cite{DLS}, and hence $B_n$ cannot embed into $V$ for~$n\geq 4$.  Note that Boone--Higman embeddings for right-angled Artin groups are now known (see Theorem~\ref{thrm:summary}), but it remains an open question whether braid groups can be embedded into finitely presented simple groups (see Problem~\ref{prob:summary}).

Based on these newer results, it is now clear that the possible subgroups of $V$ are fairly restricted.  However, the following question about subgroups of $V$ remains open.

\begin{question}
Do any one-ended hyperbolic groups embed into~$V$?  For example, does $V$ contain the fundamental group of any closed hyperbolic surface?
\end{question}

\subsection{Scott's groups and R\"over--Nekrashevych groups}\label{Subsec:Roever}
By the results of Higman and R\"over, $V$ is not a container for arithmetic groups or for groups of Burnside type.  In 1984, Elizabeth Scott solved this problem for arithmetic groups by constructing a new family of finitely presented simple groups $\mathrm{Sc}(n)$ that contain $\mathrm{GL}_n(\Z)$.

\begin{theorem}[Scott 1984 \cite{Scott2}]\label{thrm:scott}
For $n\geq 1$, the group $\mathrm{GL}_n(\Z)$ embeds into a finitely presented simple group.
\end{theorem}
\begin{proof}[Sketch of Proof]
For fixed $n\geq 2$, let $\mathcal{T}(\Z^n,2)$ be the rooted $2^n$-ary tree of all cosets of the subgroups $\Z^n\geq (2\Z)^n \geq (4\Z)^n \geq \cdots$. 
Then Scott's group  $\mathrm{Sc}(n)$ is the group of all bijections $f\colon \Z^n\to \Z^n$ with the following property: there exist partitions $A_1,\ldots,A_m$ and $B_1,\ldots,B_m$ of $\Z^n$ into cosets from $\mathcal{T}(\Z^n,2)$ such that $f$ maps each $A_i$ to $B_i$ by an affine function.   Note that when $m=1$ the resulting bijection can be any element of the affine group $\mathrm{Aff}(\Z^n)=\Z^n\rtimes \mathrm{GL}_n(\Z)$, so this is a subgroup of $\mathrm{Sc}(n)$. Scott proved that $\mathrm{Sc}(n)$ is finitely presented and simple.
\end{proof}

\begin{remark}
The group $\mathrm{Sc}(n)$ acts by homeomorphisms on the Cantor space of ends of $\mathcal{T}(\Z^n,2)$, which can be naturally identified with the set $\Z_{(2)}^n$ of all $n$-tuples of \mbox{$2$-adic} integers.  Because $\mathrm{Sc}(n)$ is a R\"over--Nekrashevych group (see below) and $\mathrm{Aff}(\Z^n)$ has type~$\mathrm{F}\!_\infty$ (see \cite[Section~VIII.9]{BrownCohomologyBook}), it follows from a result of Rachel Skipper and the fourth author that $\mathrm{Sc}(n)$ has type~$\mathrm{F}\!_\infty$~\cite{SkZa}.
\end{remark}

Scott also considered a vast generalization of this construction~\cite{Scott1}. In the terminology of Nekrashevych~\cite{NekBook}, a group $G\leq \Aut(\mathcal{T}_d)$ of automorphisms of the infinite rooted $d$-ary tree $\mathcal{T}_d$ is \newword{self-similar} if the image of $G$ under the natural isomorphism $\Aut(\mathcal{T}_d)\to \Aut(\mathcal{T}_d)^d\rtimes S_d$ lies in $G^d\rtimes S_d$.   Given a self-similar group~$G$, Scott constructed a group $V_d(G)$ of homeomorphisms of the Cantor space of ends of $\mathcal{T}_d$, namely the group generated by the Higman--Thompson group $V_d=V_{d,1}$ and an isomorphic copy of $G$ acting on a level-one subtree of $\mathcal{T}_d$.  By now, these groups $V_d(G)$ have come to be known as \newword{R\"over--Nekrashevych groups}.  Scott's groups $\mathrm{Sc}(n)$ are one example of these, where $d=2^n$ and $G$ is the affine group $\mathrm{Aff}(\Z^n)\cong \Z^n\rtimes \mathrm{GL}_n(\Z)$\footnote{Though this is implicit in Scott's work, the self-similar action of $\mathrm{Aff}(\Z^n)$ on $\mathcal{T}_{2^n}$ was later described explicitly by Brunner and Sidki \cite{BrunnerSidki}.}.

Scott proved that $V_d(G)$ always has simple commutator subgroup, and that $V_d(G)$ is finitely presented whenever $G$ is.  In 1999, Claas R\"over used Scott's construction to give the first embedding of a group of Burnside type into a finitely presented simple group.

\begin{theorem}[R\"over 1999 \cite{RoverPhD,Rov}]
If $G\leq \Aut(\mathcal{T}_2)$ is Grigorchuk's group, then $V_2(G)$ is finitely presented and simple.
\end{theorem}

Note that Grigorchuk's group is not itself finitely presented, so it is surprising that $V_2(G)$ would be.

In 2004, Volodymyr Nekrashevych placed the groups $V_d(G)$ in the context of the existing theory of self-similar groups, and observed some important connections between these groups and $C^*$-algebras~\cite{Nek}.  He also gave an algorithm to compute the abelianization of $V_d(G)$.  In particular, the group $V_d(G)$ is not always simple, though it is in Scott's and R\"over's examples. 
 Subsequent work has established additional finiteness properties for several classes of these groups \cite{BeMa,Nek2,SkZa,BHM2}, including the  first examples of simple groups with arbitrary finiteness lengths~\cite{SWZ}.

\subsection{The group $V\!\mathcal{A}$}\label{sec:VA}

In \cite{BHM}, the first author, James Hyde, and the third author gave the first explicit description of a finitely presented group $\overline{T}$ that contains the additive group $\Q$ of the rational numbers.  Specifically, $\overline{T}$ is the group of all homeomorphisms $f\colon \R\to \R$ that are lifts of elements of Thompson's group $T$, i.e.\ 
 $f$ lies in $\overline{T}$ if there exists an element $g\in T$ so that $p\circ f= g\circ p$, where $p\colon \R\to S^1$ is the universal covering. This group $\overline{T}$ was introduced by Ghys and Sergiescu in 1987 as part of their investigation into the cohomology of~$T$ \cite{GhSe}.

 The group $\overline{T}$ is contained in a larger finitely presented group $\mathcal{A}$ introduced by Brin, which has index two in the automorphism group of Thompson's group~$F$~\cite{Brin0}.  There is a natural action of $\mathcal{A}$ on the Cantor set, and in upcoming work the first author, Hyde, and the third author consider the group $V\!\A$ of homeomorphisms of the Cantor set generated by $\mathcal{A}$ and Thompson's group~$V$.  
 
\begin{theorem}[Belk--Hyde--Matucci 2023 \cite{BHM2}]\label{thm:VA}
Every countable abelian group embeds into a finitely presented simple group, namely the group $V\!\A$.
\end{theorem}
\begin{proof}[Sketch of Proof]
Just as Thompson's group $V$ contains the direct sum $\bigoplus_\omega V$, the group $\VA$ contains $\bigoplus_\omega \VA$.  As mentioned in Theorem~\ref{thm:GroupsVContains}, it is possible to embed $\Q/\Z$ into $V$ and hence into $\VA$, and $\Q\leq \overline{T}\leq \A\leq \VA$ as described above.  It follows that $\VA$ contains $\bigl(\bigoplus_\omega \Q\bigr)\oplus\bigl(\bigoplus_\omega \Q/\Z\bigr)$, and therefore $\VA$ contains every countable abelian group.  It is proven in \cite{BHM2} that $\VA$ is simple and finitely presented---indeed it has type~$\Finfty$.
\end{proof}

Note here that the class of all countable abelian groups is very complicated.  For example, the isomorphism problem for countable, torsion free abelian groups is \mbox{$\Sigma_1^1$-complete}~\cite{DoMo}.  Thus the group $\VA$ has a very intricate subgroup structure.

\begin{remark}Despite containing many more abelian groups than $V$, there are still plenty of examples of (non-abelian) groups with solvable word problem that fail to embed into $\VA$.  As one quick concrete example, we claim that $\mathrm{SL}_n(\Z)$ cannot embed in $\VA$ for $n\ge 3$. The key to seeing this is that any action of any finite index subgroup of $\mathrm{SL}_n(\Z)$ ($n\ge 3$) on any CAT(0) cube complex must fix a point~\cite{Cornulier}. Now using the CAT(0) cube complex from \cite{BHM2} on which $\VA$ acts, we see that if $\mathrm{SL}_n(\Z)$ embeds in $\VA$ then it virtually embeds in a vertex stabilizer, which is virtually an extension of one subgroup of $V$ by another subgroup of $V$. Then repeatedly using the CAT(0) cube complex from \cite{FarleyTV} on which $V$ acts with finite stabilizers, we conclude that no such embedding can exist.
\end{remark}

\subsection{Twisted Brin--Thompson groups}

In 2004, Matthew Brin introduced a family of higher-dimensional Thompson groups~$sV$ ($s\ge 1$) \cite{BrinHigherDimensional}, which we refer to as \newword{Brin--Thompson groups}.  In particular, if $V$ acts on the Cantor set $\mathfrak{C}$, then $sV$ is the group of all homeomorphisms of the Cantor cube $\mathfrak{C}^s$ that locally agree with elements of $V^s$.  All of these groups are simple \cite{BrinHigherDimensional,BrinBakers} and finitely presented~\cite{BrinPresentations,HeMa}.  Indeed, they all have type $\mathrm{F}\!_\infty$~\cite{KoMPNu,FMWZ}.  For $s\geq 2$, these groups have solvable word problem  but unsolvable torsion problem~\cite{BeBl} and unsolvable conjugacy problem~\cite{SaloUnsolvable}, making them perhaps the most natural examples of groups with these properties.

In 2022, the first and fourth authors generalized Brin's groups as follows~\cite{BZ}.  Given any group $G$ of permutations of $\{1,\ldots,s\}$, the wreath product $V\wr G=V^s\rtimes G$ acts on $\mathfrak{C}^s$ in a natural way.  In this case, the associated \newword{twisted Brin--Thompson group} $sV_G$ is the group of all homeomorphisms of $\mathfrak{C}^n$ that locally agree with elements of $V\wr G$.  This definition can be generalized further to the case where $G$ is a group of permutations of an infinite set $S$, with the associated twisted Brin--Thompson group $SV_G$ being the group of all homeomorphisms of $\mathfrak{C}^S$ that locally agree with elements of the restricted wreath product $V\wr G =(\bigoplus_S V) \rtimes G$.

One particularly important example is when $S=G$, and $G$ acts on itself by translation.

\begin{theorem}[Belk--Zaremsky 2022 \cite{BZ}]\label{thm:BZTheorem}
If $G$ is a finitely generated group then $GV_G$ is a finitely generated simple group, and $G$ embeds isometrically into $GV_G$.
\end{theorem}

It is easy to see that $GV_G$ has solvable word problem if and only if $G$ does.  Since a finitely generated simple group has solvable word problem if and only if it is computably presented, this recovers Thompson's result (Theorem~\ref{thm:Tho}) that every finitely generated group with solvable word problem embeds into a finitely generated simple group with a computable presentation. Indeed, we have the following.

\begin{corollary}\label{cor:BZCor}
A finitely generated group $G$ has solvable word problem if and only if the group $GV_G$ is computably presented.\qed
\end{corollary}

The twisted Brin--Thompson construction can also be used to produce finitely presented simple groups. We say that a group $G$ acting on a set $S$ has \newword{finitely many orbits of pairs} if the induced action of $G$ on $S\times S$ has finitely many orbits.  The following theorem is an improved version of a theorem of the first and fourth authors from~\cite{BZ}, and of the fourth author from~\cite{ZaremskyTaste}.

\begin{theorem}[Zaremsky 2025 \cite{ZaremskyFP}]\label{thm:Z}
Let $G$ be a group acting faithfully on a set~$S$. Suppose that $G$ is finitely presented, the stabilizer of each point in $S$ is finitely generated, and the action of $G$ on $S$ has finitely many orbits of pairs.  Then $G$ embeds as a subgroup of a finitely presented simple group, namely the twisted Brin--Thompson group $SV_G$.
\end{theorem}

Note that the action of an infinite group $G$ on itself by translation never has finitely many orbits of pairs, so the examples of the form $GV_G$ are not useful in this context.

This theorem has consequences for the theory of self-similar groups.  As defined by Nekrashevych, a self-similar group $G$ is \newword{contracting} if it has a finite ``nucleus'' of elements such that all but finitely many restrictions of any element of $G$ lie in the nucleus (see~\cite[Definition~2.11.1]{NekBook}). Such groups are central to the theory of iterated monodromy groups and limit spaces developed by Nekrashevych.

\begin{corollary}[Belk--Bleak--Matucci--Zaremsky 2023 \cite{BBMZ,BeMa2}]\label{cor:Contracting}
Every contracting self-similar group embeds into a finitely presented simple group.
\end{corollary}
\begin{proof}[Sketch of Proof]
Nekrashevych proved that if $G\leq \Aut(\mathcal{T}_d)$ is a contracting self-similar group then the associated R\"over--Nekrashevych group $V_d(G)$ is finitely presented~\cite{Nek2}. It is not hard to prove that $V_d(G)$ acts with finitely many orbits of pairs on the orbit of~$\overline{0}$ in the Cantor space of ends of $\mathcal{T}_d$, and that the stabilizer of any point in this orbit is finitely generated.  By Theorem~\ref{thm:Z}, it follows that $V_d(G)$ embeds into a finitely presented twisted Brin--Thompson group, and hence $G$ does as well.
\end{proof}

We should remark that we are unaware of any examples of finitely generated groups with solvable word problem that do not embed into a finitely presented twisted Brin--Thompson group, and it is possible that the family of finitely presented (simple) twisted Brin--Thompson groups is a universal family of Boone--Higman containers. Indeed, the first and fourth authors together with Francesco Fournier-Facio and James Hyde recently proved that this family is universal among all finitely presented simple groups admitting faithful highly transitive actions \cite{BFFHZ}.

\subsection{Boone--Higman for hyperbolic groups}

A finitely generated group $G$ is said to be \newword{hyperbolic} if its Cayley graph with respect to some (equivalently any) finite generating set is a hyperbolic metric space.  What this means is that the Cayley graph satisfies Gromov's thin triangles condition: there exists a $\delta>0$ so that for every triangle $(\gamma_1,\gamma_2,\gamma_3)$ of geodesic paths in the Cayley graph, each $\gamma_i$ lies in the $\delta$-neighborhood of the union of the other two. Hyperbolic groups were introduced by Gromov in 1987~\cite{Gromov}, and have since become a cornerstone of geometric group theory.  See  \cite[Part~III]{BH} for a general introduction.

The class of hyperbolic groups include many groups of interest, such as free groups, free products of finite groups, groups satisfying certain small cancellation conditions, and fundamental groups of compact hyperbolic manifolds.  In addition, Gromov argued that a ``random'' finitely presented group is hyperbolic with probability approaching~$1$, a conjecture later verified independently by Champetier~\cite{Cham} and  Ol'shanskii~\cite{Ol}. 

Dehn proved that the fundamental group of a compact hyperbolic surface always has solvable word problem~\cite{DehnAlgorithm}, through a method now known as \newword{Dehn's algorithm}~\cite[Section~III.$\Gamma$.2]{BH}.  Indeed, it turns out that every hyperbolic group has a special type of presentation called a \newword{Dehn presentation} (see~\cite[Theorem~III.$\Gamma$.2.6]{BH}), and given such a presentation a generalization of Dehn's algorithm provides a solution to the word problem, in fact in linear time.

Very recently, the authors proved the following theorem.

\begin{theorem}[Belk--Bleak--Matucci--Zaremsky 2023 \cite{BBMZ}]\label{thm:Hyperbolic}Every hyperbolic group embeds into a finitely presented simple group.
\end{theorem}

The basic strategy of proof is the following.  For any locally finite, connected graph $\Gamma$, there is an associated totally disconnected metrizable space $\partial_h\Gamma$, known as the \newword{horofunction boundary} of~$\Gamma$ (see~\cite[Chapter~II.8]{BH}).  Automorphisms of $\Gamma$ induce homeomorphisms of the horofunction boundary.  In particular, any finitely generated group $G$ acts by homeomorphisms on the horofunction boundary $\partial_h G$ of its Cayley graph.  In the case of a hyperbolic group, the horofunction boundary $\partial_h G$ is typically homeomorphic to a Cantor set, and there is a finite-to-one quotient map from $\partial_h G$ to the usual Gromov boundary of the group.

Given any finitely generated group $G$ with horofunction boundary $\partial_h G$, one can consider the group $[[\,G\mid \partial_h G\,]]$ of all homeomorphisms of $\partial_h G$ that piecewise agree with elements of~$G$.  That is, a homeomorphism $f\colon \partial_h G\to \partial_h G$ lies in $[[\,G\mid \partial_h G\,]]$ if there exists a partition of $\partial_h G$ into finitely many clopen sets such that $f$ agrees with some element of $G$ on each set of the partition.  Assuming the action of $G$ on $\partial_h G$ is faithful, the resulting group contains~$G$.  For ``nice'' hyperbolic groups $G$, the group $[[\,G\mid \partial_h G\,]]$ has the structure of a so called \newword{rational similarity group (RSG)}, introduced in \cite{BBMZ}, that is \newword{full} and \newword{contracting}.  This follows work of the first three authors, in 2021 ~\cite{BBM2}, proving that any hyperbolic group $G$ embeds as a subgroup of the asynchronous rational group $\mathscr{R}$ introduced by Grigorchuk, Nekrashevych, and Sushchanski\u\i\ (see ~\cite{GNS}).  In particular, it is possible to assign addresses to points in the horofunction boundary $\partial_h G$ so that the action of $G$ is by asynchronous automata.

One key step in the proof of Theorem~\ref{thm:Hyperbolic} is the following.

\begin{theorem}[Belk--Bleak--Matucci--Zaremsky 2023 \cite{BBMZ}]\label{thm:Hyperbolic2}Every full, contracting RSG is finitely presented, and embeds into a finitely presented simple twisted Brin--Thompson group.
\end{theorem}

Not every hyperbolic group satisfies the conditions for $[[\,G\mid \partial_h G\,]]$ to be a full, contracting RSG, but if $G$ is a non-trivial hyperbolic group then the free product $G*\Z$ is always a hyperbolic group that satisfies these hypotheses, and of course $G$ embeds into $G*\Z$, so Theorem~\ref{thm:Hyperbolic} follows from Theorem~\ref{thm:Hyperbolic2}.

The more difficult part of Theorem~\ref{thm:Hyperbolic2} is proving that full, contracting RSGs are finitely presented.  This involves a generalization of the theorem of Nekrashevych~\cite{Nek2} that the R\"over--Nekrashevych group associated to a contracting self-similar group is always finitely presented, though we should emphasize that R\"over--Nekrashevych groups are only a very special case of RSGs.  As for the other part of Theorem~\ref{thm:Hyperbolic2}, just like in Corollary~\ref{cor:Contracting}, it is not difficult to show that such groups satisfy the hypotheses of Theorem~\ref{thm:Z}, and hence embed into an associated finitely presented simple Brin--Thompson group.

\section{Summary of the state of the art}\label{sec:summary}

Having discussed many examples of finitely presented simple groups, and important groups that embed into them, let us now collect a list of prominent groups with solvable word problem for which the Boone--Higman conjecture is known to be true.

\begin{theorem}\label{thrm:summary}
The following groups embed into finitely presented simple groups:
\begin{enumerate}
\item The groups $\mathrm{GL}_n(\Z)$.\smallskip
\item Virtually nilpotent and virtually polycyclic groups.\smallskip
\item Right-angled Artin groups.\smallskip
\item Coxeter groups.\smallskip
\item Limit groups.\smallskip
\item Graph braid groups. \smallskip
\item Fundamental groups of compact 3-manifolds that admit a Riemannian metric of nonpositive curvature.
\smallskip
\item Fundamental groups of finite-volume hyperbolic 3-manifolds. \smallskip
\item One-relator groups with torsion.\smallskip
\item Countable abelian groups. \smallskip
\item Countable locally finite groups. \smallskip
\item Contracting self-similar groups and finitely presented self-similar groups.\smallskip
\item Hyperbolic groups.
\end{enumerate}
\end{theorem}

\begin{proof}
Statement (1) is Theorem~\ref{thrm:scott}. Now note that any group that virtually embeds into $\mathrm{GL}_n(\Z)$ embeds into $\mathrm{GL}_m(\Z)$ for some $m\geq n$.  This follows from the Krasner--Kaloujnine embedding theorem together with the fact that $\mathrm{GL}_n(\Z)\wr S_k$ embeds into $\mathrm{GL}_{nk}(\Z)$ for all $n$ and~$k$ using block permutation matrices. Statement (2) now follows from the facts that every finitely generated nilpotent group is polycyclic \cite[Theorem 17.2.2]{KargMerz} and that every polycyclic group embeds into $\mathrm{GL}_n(\Z)$ \mbox{\cite{Auslander,Swan}}.  Statement (3) follows from a theorem of Hsu and Wise that finitely generated right-angled Artin groups embed into $\mathrm{SL}_n(\Z)$ \cite{HsuWise}.  Statements (4) through (9) now follow from known results about groups that embed or virtually embed into right-angled Artin groups (see \cite{HaWi2}, \cite{Wise1}, \cite{CrWi}, \cite{PyWi}, \cite{Agol}, and \cite{Wise1}, respectively). Statement (10) is Theorem~\ref{thm:VA}, and statement (11) is in item (3) of Theorem~\ref{thm:GroupsVContains}. Finally, (12) follows from \cite{BBMZ} (see also \cite{BeMa2}) and \cite{ZarSelfSim}, and (13) follows from \cite{BBMZ}.
\end{proof}

\begin{remark}
The first three authors have shown that every right-angled Artin group embeds into some Brin--Thompson group $nV$~\cite{BBM1}, and Salo has recently proven that indeed they all embed into~$2V$~\cite{Salo2V}. It follows that all of the groups (3) through (9) above embed into the single finitely presented simple group~$2V$.
\end{remark}

Now let us collect some concrete examples of groups and families of groups of interest that have solvable word problem but for which, to the best of our knowledge, the Boone--Higman conjecture remains open in general. Some of these overlap or are contained in each other, but all of them are prominent enough to deserve special mention. Also note that some of these have been solved since this survey was first written: see Remark~\ref{rmk:solved}.

 \begin{problem}\label{prob:summary}
    Prove the Boone--Higman conjecture for:
    \begin{enumerate}
        \item Braid groups.\smallskip
        \item Mapping class groups of surfaces.\smallskip
        \item The groups $\mathrm{Aut}(F_n)$ and $\mathrm{Out}(F_n)$.\smallskip
        \item Non-solvable Baumslag--Solitar groups $BS(m,n)$, for example $BS(2,3)$\footnote{Bartholdi and Sunik have shown that solvable Baumslag--Solitar groups $BS(1,n)$ are self-similar~\cite{BaSu}, and they can be seen to embed in finitely presented simple groups, e.g.~following \cite{ZaremskyDehn}.}.\smallskip
        \item $\mathrm{GL}_n(\Q)$.\smallskip
        \item Free Burnside groups $B(m,n)$ with solvable word problem.\smallskip
        \item Finitely presented metabelian groups\footnote{By a theorem of Baumslag and Remeslennikov, every finitely generated metabelian group embeds into a finitely presented one~\cite{Baumslag,Rem}.  Such groups have solvable word problem~\cite{Baumslag2}, though the same does not hold for solvable groups of derived length three~\cite{Har}.}.\smallskip
        \item Free-by-cyclic groups.\smallskip
        \item One-relator groups (without torsion).\smallskip
        \item CAT(0) groups.\smallskip
        \item Automatic groups.\smallskip
        \item Finitely presented residually finite groups.\smallskip
        \item Artin groups with solvable word problem\footnote{It is an open question whether all Artin groups have solvable word problem \cite[Problem~10]{Charney}, but this is known for several large classes, such as right-angled Artin groups (see Theorem~\ref{thrm:summary}) and Artin groups of spherical \cite{Deligne} or Euclidean \cite{McSu} type.}.\smallskip
    \end{enumerate}
 \end{problem}

\begin{remark}\label{rmk:solved}
We should mention that since this survey was first written, some of the items on this list have been resolved. Let us collect this recent progress here. First, Kai-Uwe Bux, Claudio Llosa Isenrich, and Xiaolei Wu in \cite{BLIW} proved the Boone--Higman conjecture for all Baumslag--Solitar groups and all free-by-cyclic groups, so items (4) and (8) are fully handled. Second, the first and fourth author, together with Francesco Fournier-Facio and James Hyde in \cite{BFFHZ}, proved the conjecture for all $\Aut(F_n)$, which also proves it for braid groups and for mapping class groups of surfaces with non-empty boundary or at least one puncture, as well as Artin groups of types $A_n$ (the braid groups), $B_n=C_n$, $D_n$, $I_2(m)$, and $\widetilde{A}_n$. Thus item~(1) is fully handled, and items (2), (3), and (13) are partially handled.  Another case of item (13) was also handled in \cite{BLIW}, namely Euclidean triangle Artin groups.  Item (2) remains open for mapping class groups of closed surfaces of genus~$3$ or greater, item (3) remains open for $\Out(F_n)$ for all $n\geq 3$, and item (13) remains open for many classes of Artin groups, including those of exceptional spherical type and those of Euclidean type other than the~$\widetilde{A}_n$ and triangle cases ($\widetilde{C}_2$ and $\widetilde{G}_2$).
\end{remark}

\begin{remark}One possible approach to item (2) for mapping class groups of closed hyperbolic surfaces involves the action of the mapping class group on the Thurston boundary of Teichm\"uller space. This boundary is homeomorphic to a sphere of dimension $6g-7$, with coordinate charts given by complete train tracks, and William Thurston observed that the mapping class group acts by piecewise-integral-projective (PIP) homeomorphisms with respect to these coordinates~\cite{Thurston}.  Thurston also observed that the group of all PIP homeomorphisms of a circle is isomorphic to Thompson's group~$T$ (see \cite{Greenberg} or \cite[\S 7]{CFP}), and asked whether the groups of PIP homeomorphisms of higher-dimensional spheres are finitely generated~\cite[pg.~193]{Thurston2}.  This question remains open, but if these groups turn out to be finitely presented simple groups, this would yield a natural Boone--Higman embedding for mapping class groups of closed surfaces.
\end{remark}

Finally, we should mention a related question, posed by Leary in \cite[Section~21]{Leary}, which asks whether every finitely presented group embeds into a group with type~$\mathrm{F}\!_3$, or more generally whether every type $\mathrm{F}\!_n$ group embeds into a type $\mathrm{F}\!_{n+1}$ group. One could similarly ask whether every finitely presented group embeds in a type $\mathrm{F}\!_\infty$ group.

\begin{question}
Does every finitely presented group embed as a subgroup of a type $\mathrm{F}\!_\infty$ group?
\end{question}

There do exist type $\mathrm{F}\!_\infty$ groups with unsolvable word problem~\cite{CollinsMiller}, so this is not an impediment. If this question and the Boone--Higman conjecture both have positive answers, then one could further ask whether every finitely generated group with solvable word problems embeds in a simple group of type $\mathrm{F}\!_\infty$.

\bigskip
\newcommand{\arXiv}[1]{\href{https://arxiv.org/abs/#1}{arXiv}}
\newcommand{\doi}[1]{\href{https://doi.org/#1}{Crossref\,}}
\bibliographystyle{plain}

\end{document}